\documentclass[12pt,a4paper]{article}

\usepackage[utf8]{inputenc}
\usepackage[T2A]{fontenc}

\usepackage{amsmath,amssymb,amsthm,eucal,mathrsfs}
\usepackage{hyperref}
\usepackage{todonotes}
\usepackage[margin=2cm]{geometry}
\usepackage[pagewise]{lineno}
\newtheorem{theorem}{Theorem}[section]
\newtheorem{lemma}[theorem]{Lemma}
\newtheorem{proposition}[theorem]{Proposition}
\newtheorem{corollary}[theorem]{Corollary}

\theoremstyle{definition}

\theoremstyle{remark}
\newtheorem*{remark*}{Remark}

\newcommand{\BB}{\mathscr{B}}

\newcommand{\MM}{{\mathfrak{M}}}
\newcommand{\NN}{{\mathfrak{N}}}
\newcommand{\TT}{{\mathfrak{T}}}

\title{Uniform Continuity in Distribution for Borel Transformations of Random Fields}

\author{Alexander~I.~Bufetov\thanks{Steklov Mathematical Institute of the Russian Academy of Sciences.\\
		This work was performed at the Steklov International Mathematical Center and supported by the Ministry of
		Science and Higher Education of the Russian Federation (agreement no.~075-15-2025-303).}}

\date{}

\begin{document}
	
	\maketitle
	
	\begin{flushright}
		{\it{Моему дорогому научному руководителю\\
		Якову Григорьевичу Синаю\\
		с любовью и благодарностью.}}
	\end{flushright}
	
	\begin{abstract}
		Simple sufficient conditions are given that ensure the uniform continuity in distribution for Borel transformations of random fields.
	\end{abstract}
	
	\section{Introduction and the statement of the main result}
	
	Let $W$ be a complete separable metric space, let $\MM_1(W)$ be the space  of
	Borel probability measures on $W$ endowed with the weak topology and let $\NN$ be a compact subset of the space $\MM_1(W)$. Assume that $V \subset W$ is a Borel
	subset satisfying $\eta(V) = 1$ for any measure $\eta \in \NN$.
	
	A Borel mapping
	\begin{equation}
		g \colon \NN \times V \to W
		\label{eq:g}
	\end{equation}
	induces a Borel mapping
	\begin{equation}
		g_* \colon \NN \to \MM_1(W)
		\label{eq:pushforward-def}
	\end{equation}
	by the formula
	\begin{equation}
		g_* \eta = (g(\eta, \cdot))_* \eta.
		\label{eq:pushforward}
	\end{equation}
	The map $g_*$ is equivalently given by the formula
	\begin{equation}
		\int_W \varphi \, d g_* \eta = \int_W \varphi(g(\eta, w)) \, d\eta(w)
		\label{eq:phi}
	\end{equation}
	that holds for any bounded continuous function $\varphi \colon W \to \mathbb{C}$.
	
	The aim of this note is to give convenient sufficient conditions ensuring that the correspondence
	$g \mapsto g_*$ induces a uniformly continuous map. Of particular interest to us is the case where the space $W$
	is the space of realizations of a random field; under additional assumptions on the
	regularity of the random field, the space $W$ is a complete separable metric space (see Section~\ref{sec:4} below).
	
	The space $\MM_1(W)$ can be turned into a metric space in several different natural ways. For our purposes  the L\'evy--Prokhorov metric is particularly convenient, and we recall that the weak topology on $\MM_1(W)$ is induced by the L\'evy--Prokhorov
	metric $d_{\mathrm{LP}}$, which is uniquely defined by the following requirement (cf. Bogachev \cite{Bogachev}, Billingsley \cite{Billingsley}): the
	L\'evy--Prokhorov distance between measures $\eta_1, \eta_2 \in \MM_1(W)$ does not exceed $\varepsilon$ if
	and only if for every closed set $K \subset W$ and its $\varepsilon$-neighbourhood
	\[
	K_\varepsilon = \{ w \in W : d(w, K) \le \varepsilon \}
	\]
	we have
	\[
	\eta_1(K) \le \eta_2(K_\varepsilon) + \varepsilon, \qquad
	\eta_2(K) \le \eta_1(K_\varepsilon) + \varepsilon.
	\]
	
	We endow the space of Borel mappings \eqref{eq:g} with the metric of uniform convergence in probability, and
	the space of Borel mappings from $\NN$ to $\widetilde{\NN}$ with the  Tchebycheff uniform metric. We fix a compact set $\widetilde{\NN} \subset \MM_1(W)$ and only consider maps $g$ of the form \eqref{eq:g} satisfying the additional requirement
	$g_*(\NN) \subset \widetilde{\NN}$. The correspondence $g \mapsto g_*$ is then uniformly continuous with respect to our metrics. We 
	proceed to the precise formulations.
	
	We fix a metric $d$  on the space $W$, and we  write $d_{\mathrm{LP}}$ for the corresponding L\'evy--Prokhorov metric on the space $\MM_1(W)$. For a compact subset $\NN \subset \MM_1(W)$ and a Borel subset $V \subset W$ satisfying
	$\eta(V) = 1$ for any $\eta \in \NN$, introduce the space $\BB(\NN \times V, W)$ of Borel maps $$g \colon \NN \times V \to W.$$
	We endow the space $\BB(\NN \times V, W)$ with the topology $\TT_{\mathrm{prob}}$ generated by the neighbourhoods
	\[
	U(g_0, \varepsilon, \delta) = \left\{
	g \in \BB(\NN \times V, W) \,:\, \eta(\{w \in W : d(g(\eta, w), g_0(\eta, w)) > \varepsilon\}) < \delta \text{ for all } \eta \in \NN
	\right\}
	\]
	for all $\varepsilon > 0$ and $\delta > 0$. The collection of the sets $U(g_0, \varepsilon, \delta)$ over all $\varepsilon>0, \delta >0$ forms a subbasis for the topology $\TT_{\mathrm{prob}}$ . The topology  $\TT_{\mathrm{prob}}$ on the space $\BB(\NN \times V, W)$is metrizable: one directly verifies
	that the distance
	\[
	d(g_1, g_2) =
	\sum_{l=1}^{\infty}
	2^{-l}
	\frac{\displaystyle\sup_{\eta \in \NN}
		\eta(\{w \in W : d(g_1(\eta, w), g_2(\eta, w)) > 2^{-l}\})}
	{1 + \displaystyle\sup_{\eta \in \NN}
		\eta(\{w \in W : d(g_1(\eta, w), g_2(\eta, w)) > 2^{-l}\})}
	\]
	induces the topology $\TT_{\mathrm{prob}}$ on $\BB(\NN \times V, W)$.
	
	Now let $\BB(\NN, \MM_1(W))$ stand for the space of all Borel maps $h \colon \NN \to \MM_1(W)$. We endow the space $\BB(\NN, \MM_1(W))$ with the
	Tchebycheff uniform metric with respect to the L\'evy--Prokhorov metric on the space $\MM_1(W)$. For a compact set
	$\widetilde{\NN} \subset \MM_1(W)$ let $\BB(\NN, \widetilde{\NN})$ be the subspace of Borel maps
	$h \colon \NN \to \MM_1(W)$ satisfying the inclusion $h(\NN) \subset \widetilde{\NN}$. We note that the subspace $\BB(\NN, \widetilde{\NN})$
	is closed in the Tchebycheff uniform topology.
	
	Let $\widetilde{\NN} \subset \MM_1(W)$ be a compact set. Let $\BB_{\widetilde{\NN}}(\NN \times V, W)\subset \BB(\NN \times V, W)$ be the
	subspace of maps $g$ such that the map $g_*$ defined by \eqref{eq:pushforward}
	satisfies the inclusion $g_*(\NN) \subset \widetilde{\NN}$. The subspace $\BB_{\widetilde{\NN}}(\NN \times V, W) \subset \BB(\NN \times V, W)$
	is closed by definition.
	The main result of this paper is
	
	\begin{theorem}
		\label{thm:main}
		The correspondence $g \mapsto g_*$ induces a uniformly continuous map from the space
		$\BB_{\widetilde{\NN}}(\NN \times V, W)$ to the space $\BB(\NN, \widetilde{\NN})$.
	\end{theorem}
	
	\section{Reduction to the case $W = \mathbb{C}^{\mathbb{N}}$}
	
	Endow the space
	\[
	\mathbb{C}^{\mathbb{N}} = \{ t = (t_1, t_2, \ldots) : t_n \in \mathbb{C},\, n \in \mathbb{N} \}
	\]
	with the distance
	\[
	d(t^{(1)}, t^{(2)}) =
	\sum_{l=1}^{\infty}
	2^{-l}
	\frac{|t^{(1)}_l - t^{(2)}_l|}{1 + |t^{(1)}_l - t^{(2)}_l|}.
	\]
	
	The Stone--Weierstrass Theorem directly implies the following two propositions.
	
	\begin{proposition}
		\label{prop:2.1}
		Let $W$ be a complete separable metric space and let $K$ be a compact subset of $W$.
		There exists a countable family of $1$-Lipschitz functions $\Phi = (\varphi_n)_{n \in \mathbb{N}}$ on $K$
		such that the correspondence
		\[
		w \mapsto (\varphi_n(w))_{n \in \mathbb{N}}
		\]
		defines a bi-Lipschitz homeomorphism of $K$ onto its image.
	\end{proposition}
	
	\begin{proposition}
		\label{prop:2.2}
		Let $V$ and $W$ be  complete separable metric spaces. Let $V' \subset V$ and $\NN \subset \MM_1(V)$
		be Borel subsets such that $\eta(V') = 1$ for any $\eta \in \NN$. Let $\tau \colon V' \to W$ be a Borel map.
		Assume that for any $\varepsilon > 0$ there exist $R \ge 1$ and a Borel subset $V(\varepsilon) \subset V'$ such that
		\[
		\sup_{\eta \in \NN} \eta(V' \setminus V(\varepsilon)) < \varepsilon,
		\]
		and for any $x, y \in V(\varepsilon), x \ne y$ we have
		\[
		R^{-1} \le \frac{d_W(\tau x, \tau y)}{d_V(x, y)} \le R.
		\]
		Then the induced map $\tau_* \colon \NN \to \MM_1(W)$ is a homeomorphism onto its image.
	\end{proposition}
	
	Propositions~\ref{prop:2.1} and~\ref{prop:2.2} imply that it suffices to establish Theorem~\ref{thm:main}
	for $W = \mathbb{C}^{\mathbb{N}}$.
	
	We now assume that $\NN$ and $\widetilde{\NN}$ are compact subsets of the space $\MM_1(\mathbb{C}^{\mathbb{N}})$
	of Borel probability measures on the space
	\[
	\mathbb{C}^{\mathbb{N}} = \{ t = (t_1, t_2, \ldots) : t_n \in \mathbb{C},\, n \in \mathbb{N} \}.
	\]
	
	
	For a Borel subset $V \subset \mathbb{C}^{\mathbb{N}}$ satisfying $\eta(V) = 1$ for every $\eta \in \NN$,
	denote by $\BB(\NN \times V, \mathbb{C}^{\mathbb{N}})$ the corresponding space of Borel maps
	\[
	g \colon \NN \times V \to \mathbb{C}^{\mathbb{N}}, \qquad
	g(\eta, t) = (g^{(n)}(\eta, t))_{n \in \mathbb{N}},
	\]
	where $g^{(n)}(\eta, t) \in \mathbb{C}$ for $n \in \mathbb{N}$.
	
	The topology on $\BB(\NN \times V, \mathbb{C}^{\mathbb{N}})$ is induced by the distance function
	\[
	d(g, \widehat{g}) =
	\sum_{l,k=1}^{\infty}
	2^{-l-k}
	\frac{\displaystyle\sup_{\eta \in \NN}
		\eta(\{t \in V : |g^{(k)}(\eta, t) - \widehat{g}^{(k)}(\eta, t)| > 2^{-l}\})}
	{1 + \displaystyle\sup_{\eta \in \NN}
		\eta(\{t \in V : |g^{(k)}(\eta, t) - \widehat{g}^{(k)}(\eta, t)| > 2^{-l}\})}.
	\]
	
	By definition, a sequence $g^{(n)} \in \BB(\NN \times V, \mathbb{C}^{\mathbb{N}})$ converges to $g$ if and only if
	for every $k \in \mathbb{N}$ and every $\varepsilon > 0$ we have
	\[
	\lim_{n \to \infty}
	\sup_{\eta \in \NN}
	\eta(\{ t \in V : |g^{(n)}_k(\eta, t) - g^{\phantom{(n)}}_k(\eta, t)| > \varepsilon\})
	= 0.
	\]
	In other words, the distance $d$ induces the topology of coordinatewise convergence in probability, uniform in $\eta \in \NN$.
	
	As above, a Borel mapping $g \in \BB(\NN \times V, \mathbb{C}^{\mathbb{N}})$ (cf. \eqref{eq:phi}) induces a map 
	$$g_* \colon \NN \to \MM_1(\mathbb{C}^{\mathbb{N}})$$
	given by the formula
	\[
	g_* \eta = (g(\eta, \cdot))_* \eta.
	\]
	
	Recall that the space $\BB_{\widetilde{\NN}}(\NN \times V, \mathbb{C}^{\mathbb{N}})$ is by definition the subspace of
	$\BB(\NN \times V, \mathbb{C}^{\mathbb{N}})$ consisting of those  maps $g$ whose corresponding induced map $g_*$  satisfies the inclusion $g_*(\NN) \subset \widetilde{\NN}$.
	
	As before, we endow the space $\BB(\NN, \widetilde{\NN})$ with the Tchebycheff uniform metric with respect to the
	L\'evy--Prokhorov metric on $\MM_1(\mathbb{C}^{\mathbb{N}})$. The subspace $\BB_{\widetilde{\NN}}(\NN \times V, \mathbb{C}^{\mathbb{N}})$
	is closed in $\BB(\NN \times V, \mathbb{C}^{\mathbb{N}})$. The special case of Theorem~\ref{thm:main} in our particular setting is the following
	
	\begin{theorem}
		\label{thm:2.3}
		The correspondence $g \mapsto g_*$ defines a uniformly continuous map from $\BB_{\widetilde{\NN}}(\NN \times V, \mathbb{C}^{\mathbb{N}})$ to $\BB(\NN, \widetilde{\NN})$.
	\end{theorem}
	
	As we have seen, Theorem~\ref{thm:2.3} directly implies Theorem~\ref{thm:main}.
	
	\begin{corollary}
		\label{cor:2.4}
		Let Borel maps $g^{(n)} \in \BB(\NN \times V, \mathbb{C}^{\mathbb{N}})$, $n \in \mathbb{N}$, be such that the map $(g^{(n)})_*$ is
		continuous for every $n \in \mathbb{N}$ and $g^{(n)} \to g$ in $\BB(\NN \times V, \mathbb{C}^{\mathbb{N}})$,
		then $g_*$ is uniformly continuous.
	\end{corollary}
	
	Let a subset $S_0 \subset \mathbb{N}$ and a compact subset $\NN \subset \MM_1(\mathbb{C}^{\mathbb{N}})$ satisfy
	the following condition: for every $n \in \mathbb{N}$ and every $\delta > 0, \varepsilon > 0$ there exists
	$s_0 \in S_0$ such that
	\begin{equation}
		\label{eq:cond}
		\sup_{\eta \in \NN}
		\eta(\{ t \in \mathbb{C}^{\mathbb{N}} : |t(n) - t(s_0)| > \delta \})
		< \varepsilon.
	\end{equation}
	The above argument implies the following
	\begin{proposition}
		\label{prop:2.5}
		Under condition~\eqref{eq:cond}, the projection $\mathbb{C}^{\mathbb{N}} \to \mathbb{C}^{S_0}$ induces
		a homeomorphism of the compact set $\NN$ onto its image.
	\end{proposition}
	
	\section{Proof of Theorem~\ref{thm:2.3}}
	
	Let $\NN$ and $\widetilde{\NN}$ be compact subsets of the space $\MM_1(\mathbb{C}^{\mathbb{N}})$.
	Endow the space $C(\NN, \widetilde{\NN})$ of continuous maps from $\NN$ to $\widetilde{\NN}$ with the Tchebycheff metric:
	for $h, h' \in C(\NN, \widetilde{\NN})$ the Tchebycheff uniform distance $d(h, h')$ is defined by the formula
	\[
	d(h, h') = \max_{\eta \in \NN} d_{\mathrm{LP}}\left(h(\eta), h'(\eta)\right).
	\]
	The space $C(\NN, \widetilde{\NN})$ is a complete separable metric space.
	
	To a map $h \in C(\NN, \widetilde{\NN})$ we assign its projection onto the 
	first $l$ coordinates
	\[
	\pi_l h(\kappa) = \operatorname{distr}^l_{h(\kappa)}([1, l]),
	\]
	where the symbol $\operatorname{distr}^l_\eta$ stands for the $l$-dimensional marginal distribution of a measure $\eta$ on the first $l$ coordinates.
	
	\begin{proposition}
		\label{prop:3.1}
		Let $M$ be a complete separable metric space, and let  $F \colon M \to C(\NN, \widetilde{\NN})$ be a Borel map.  If for any $l \in \mathbb{N}$ the 
		 finite-dimensional projection mapping
		\[
		m \mapsto \pi_l F(m)
		\]
		is uniformly continuous with respect to the L\'evy--Prokhorov metric, then the map $F$ is also uniformly continuous with respect to the L\'evy--Prokhorov metric.
	\end{proposition}
	
	\begin{proof}
		A  simple compactness argument  connects the distance between measures with
		the distances between their finite-dimensional distributions.
		
		\begin{lemma}
			\label{lem:3.2}
			Let $K \subset \MM_1(\mathbb{C}^{\mathbb{N}})$ be a compact set. For every $\varepsilon > 0$ there exist $l \in \mathbb{N}$
			and $\delta > 0$, depending only on $K$, such that
			if $\eta_1, \eta_2 \in K$ satisfy the inequality
			\[
			d_{\mathrm{LP}}(\operatorname{distr}^l_{\eta_1}([1, l]), \operatorname{distr}^l_{\eta_2}([1, l])) \le \delta,
			\]
			then  we have $$d_{\mathrm{LP}}(\eta_1, \eta_2) \le \varepsilon.$$
		\end{lemma}
		
		Lemma~\ref{lem:3.2} follows  from the Prokhorov Theorem and  directly implies Proposition \ref{prop:3.1}.
		\end{proof}

	Let $\pi_l$ denote the projection $\pi_l \colon \mathbb{C}^{\mathbb{N}} \to \mathbb{C}^l$ onto the first $l$ coordinates.
	Slightly abusing notation, we use the same symbol for the corresponding projection
	$\pi_l \colon \MM_1(\mathbb{C}^{\mathbb{N}}) \to \MM_1(\mathbb{C}^l)$.
	
	Proposition~\ref{prop:3.1} implies the following simple
	
	\begin{corollary}
		\label{cor:3.3}
		Let $\NN, \widetilde{\NN} \subset \MM_1(\mathbb{C}^{\mathbb{N}})$ be compact sets and let $h_n \colon \NN \to \widetilde{\NN}$,
		$h \colon \NN \to \widetilde{\NN}$ be Borel maps. The sequence $h_n$ converges to $h$ uniformly
		if and only if for every $l \in \mathbb{N}$ the sequence $\pi_l h_n$ converges uniformly to $\pi_l h$
		as $n \to \infty$.
	\end{corollary}
	
	We now conclude the proof of Theorem~\ref{thm:2.3}. As before, let $\NN \subset \MM_1(\mathbb{C}^{\mathbb{N}})$
	be a compact subset, and let $V \subset \mathbb{C}^{\mathbb{N}}$ be a Borel subset satisfying $\eta(V) = 1$ for every $\eta \in \NN$.
	Next, let a sequence  of Borel mappings $g_n \in \BB(\NN \times V, \mathbb{C}^{\mathbb{N}})$ converge, under the topology $\TT_{\mathrm{prob}}$, to the limit $g \in \BB(\NN \times V, \mathbb{C}^{\mathbb{N}})$.
	But then, by definition, for every $l \in \mathbb{N}$ the uniform convergence
	\[
	\pi_l(g_n)_* \to \pi_l g_* \quad \text{as } n \to \infty
	\]
	holds in the space $\BB(\pi_l(\NN), \MM_1(\mathbb{C}^l))$, endowed with the Tchebycheff metric. Corollary \ref{cor:3.3} now implies Theorem~\ref{thm:2.3}.
	Theorem~\ref{thm:2.3} is proved completely. \qed
	
	\section{A sufficient condition for separability of the space of realizations}
	\label{sec:4}
	
	Let $S$ be an index set; here we mainly need the case when S is a complete separable metric space. Let $\ell_1^+, \widetilde{\ell}_1^+$ be positive continuous functions on $S$ and let $\ell_2^+, \widetilde{\ell}_2^+$
	be nonnegative continuous  functions on $S \times S$ such that $\ell_2^+(s, s) = \widetilde{\ell}_2^+(s, s) = 0$
	for all $s \in S$. To the pair of functions $\ell_1^+, \ell_2^+$ 
	we assign the family $\MM_2(\ell_1^+, \ell_2^+)$ of measures satisfying the following assumptions:
	\begin{enumerate}
		\item For all $s \in S$ we have $t(s) \in L^2(\mathbb{C}^S, \eta)$ and $\|t(s)\|_{L^2(C^S, \eta)} \le \ell_1^+(s)$;
		\item For all $s_1, s_2 \in S$ we have $\|t(s_1) - t(s_2)\|_{L^2(\mathbb{C}^S, \eta)} \le \ell_2^+(s_1, s_2)$.
	\end{enumerate}
	
	Let $V \subset {\mathbb C}^S$ be a Borel subset such that $\eta(V) = 1$ for all $\eta \in \MM_2(\ell_1^+, \ell_2^+)$. Denote
	\[
	\BB(\ell_1^+, \ell_2^+; \widetilde{\ell}_1^+, \widetilde{\ell}_2^+) =
	\BB_{\MM_2(\widetilde{\ell}_1^+, \widetilde{\ell}_2^+)}(\MM_2(\ell_1^+, \ell_2^+) \times V, {\mathbb C}^S).
	\]
	Let $S_0 = \{s_1, s_2, \ldots\} \subset S$ be a countable subset. We endow the space $\BB(\ell_1^+, \ell_2^+; \widetilde{\ell}_1^+, \widetilde{\ell}_2^+)$
	with a slightly stronger topology $\TT_2$ induced by the distance
	\[
	d_2(g, \widehat{g}) =
	\sum_{s_\ell \in S_0}
	2^{-\ell}
	\frac{\displaystyle\sup_{\eta \in \MM_2(\ell_1^+, \ell_2^+)}
		\|g(\eta, t)(s_\ell) - \widehat{g}(\eta, t)(s_\ell)\|_{L^2(\mathbb{C}^S, \eta)}}
	{1 + \displaystyle\sup_{\eta \in \MM_2(\ell_1^+, \ell_2^+)}
		\|g(\eta, t)(s_\ell) - \widehat{g}(\eta, t)(s_\ell)\|_{L^2(\mathbb{C}^S, \eta)}}.
	\]
	
	Theorem~\ref{thm:2.3} yields the following
	\begin{corollary}
		\label{cor:4.1}
		The correspondence $g \mapsto g_*$ induces a uniformly continuous map from the metric space
		$\BB(\ell_1^+, \ell_2^+; \widetilde{\ell}_1^+, \widetilde{\ell}_2^+)$ to the space $\BB(\MM_2(\ell_1^+, \ell_2^+), \MM_2(\widetilde{\ell}_1^+, \widetilde{\ell}_2^+))$ endowed with the Tchebycheff uniform metric.
	\end{corollary}
	
	From now on we let $S$ be a complete separable metric space. Endow the space
	\[
	{\mathbb C}^S = \{ t : S \to \mathbb{C} \}
	\]
	of $\mathbb{C}$-valued functions on $S$ with the Tychonoff product topology.
	
	Let $\MM_1(\mathbb{C}^S)$ be the space of Borel probability measures on 
	$\mathbb{C}^S$ endowed with the weak topology.
	Given distinct points $s_1, \ldots, s_l \in S$ the corresponding finite-dimensional distribution of a measure $\eta \in \MM_1(C^S)$ is by definition the image of $\eta$ under the projection
	\[
	t \mapsto (t(s_1), \ldots, t(s_l)).
	\]
	The symbol $\operatorname{distr}^l_\eta(s_1, \ldots, s_l)$ stands for the $l$-dimensional distribution
	of $\eta$ corresponding to the distinct points $s_1, \ldots, s_l \in S$.
	We start with a simple observation.
	
	\begin{proposition}
		\label{prop:4.2}
		A sequence of measures $\eta_n \in \MM_1(\mathbb{C}^S)$ converges weakly to a limit measure $\eta \in \MM_1(\mathbb{C}^S)$ if and only if
		for every $l \in \mathbb{N}$ and every finite collection of distinct coordinates $s_1, \ldots, s_l \in S$ the corresponding
		finite-dimensional distributions converge weakly as $n \to \infty$, that is, if we have 
		\[
		\lim\limits_{n\to\infty}\operatorname{distr}^l_{\eta_n}(s_1, \ldots, s_l)= \operatorname{distr}^l_{\eta}(s_1, \ldots, s_l).
		\]
	\end{proposition}
	
	\begin{proof}
		The Bockstein Theorem \cite{Bockstein} states that any bounded continuous function $f$ on $\mathbb{C}^S$
		only depends on a countable subcollection of coordinates:
		\[
		f(t(s)) = f(t(s_1), t(s_2), \ldots, t(s_\ell), \ldots).
		\]
		It is therefore sufficient to establish the proposition for $S = \mathbb{N}$.
		
		We now need to show that a subset $\NN_0 \subset \MM_1(\mathbb{C}^{\mathbb{N}})$ is precompact if for every $l \in \mathbb{N}$
		the set
		\[
		\NN_{0,l} = \{\operatorname{distr}^l_{\eta}(1, \ldots, l) : \eta \in \NN_0\}
		\]
		of the finite-dimensional projections of $\NN_0$ on the first $l$ coordinates is precompact.
		
		Indeed, the Prokhorov Theorem implies that for every $l \in \mathbb{N}$ there exists a compact set $K_l \subset \mathbb{C}^l$
		such that for every measure $\eta \in \NN_0$ we have
		\[
		\operatorname{distr}^l_{\eta}(1, \ldots, l)(K_l) \ge 1 - \varepsilon/2^l.
		\]
		The compact set $K_l$ is bounded in the $l$-th coordinate:
		\[
		K_l \subset \{(t_1, \ldots, t_l) : |t_l| \le R_l\}
		\]
		for some $R_l > 0$. We then have
		\[
		\eta(\{t : |t(l)| > R_l\}) < \varepsilon/2^l.
		\]
		The compact set
		\[
		K = \{ t : |t(1)| \le R_1, \ldots, |t(l)| \le R_l, \ldots \}
		\]
		then satisfies the bound
		\[
		\eta(K) \ge 1 - \sum_{l=1}^{\infty} \eta(\{ t : |t(l)| > R_l \}) \ge 1 - \varepsilon.
		\]
		and the set $\NN_0$ is precompact.
	\end{proof}
	
	By definition, the subset $\MM_2(\ell_1^+, \ell_2^+)$ is closed in $C^S$ with respect to the weak topology. Let $S_0 \subset S$ be a countable dense subset.
	
	\begin{proposition}
		\label{prop:4.3}
		Let $S_0 \subset S$ be countable and dense. The projection map $\mathbb{C}^S \to \mathbb{C}^{S_0}$ induces a homeomorphism
		of the space $\MM_2(\ell_1^+, \ell_2^+)$ onto its image.
	\end{proposition}
	
	\begin{proof}
		The projection mapping $\mathbb C^S \to \mathbb C^{S_0}$ is continuous by definition. Next, one directly verifies that our natural projection is injective in restriction to the subset $\MM_2(l_1^+, l_2^+)$. It remains to show that our natural projection map is open.
		
		Fix $\varepsilon > 0$ and let $f \colon \mathbb{C}^S \to \mathbb C$ be continuous and bounded.
		Introduce the neighbourhood $U(\varepsilon, f, \eta_0)$ of a measure $\eta_0 \in \MM_1(\mathbb{C}^S)$ by the formula
		\[
		U(\varepsilon, f, \eta_0) = \left\{ \eta \in \MM_1(C^S) :
		\left| \int f \, d\eta - \int f \, d\eta_0 \right| < \varepsilon \right\}.
		\]
		The family of neighbourhoods $U(\varepsilon, f, \eta_0)$, taken over all $\varepsilon > 0$ and all
		bounded continuous functions $f \colon C^S \to \mathbb{C}$, forms a subbasis of the weak topology on $\MM_1(\mathbb{C}^S)$.
In order to establish Proposition \ref{prop:4.3} it suffices therefore to prove that the image of any neighbourhood of the form $U(\varepsilon, f, \eta_0)$
is open under the natural projection  mapping $\mathbb C^S \to \mathbb C^{S_0}$ . By the Bockstein Theorem \cite{Bockstein} the function $f$ depends only
		on a countable subcollection of coordinates. Therefore,  Proposition \ref{prop:4.3} for general $S$ follows from the particular case of
		countable $S$ --- which  is proved in Proposition~\ref{prop:2.5}. Proposition \ref{prop:4.3} is proved completely.
	\end{proof}
	
	An important particular case of Theorem \ref{thm:main} arises when the compact set $\NN$ is defined by bounds on the exponential moments of realizations of a stochastic process and the mapping $g$ is the normalized exponential: that is, we take $\gamma \in \mathbb{R}$ or, more generally, $\gamma \in \mathbb{C}$ , let
 $(X_s)_{s \in S}$ be a stochastic process on our space $S$, and consider the correspondence
	\[
	g \colon X_s \mapsto \frac{\exp(\gamma X_s)}{\mathbb{E}_\eta \exp(\gamma X_s)}.
	\]
	
	There is an extensive literature devoted to Gaussian Multiplicative Chaos, both real and complex; see e.g. \cite{Berestycki2,Bufetov1,Bufetov2,Bufetov3,Chhaibi,FyodorovBouchaud,FyodorovKeating,FyodorovKhoruzhenkoSimm,Kahane,LacoinGMC,LacoinRRV,LacoinUniv,LambertOstrovskySimm,NajnudelEtAl,NikulaSaksmanWebb,RhodesVargas,Shamov,WebbCUE}, as well as the reviews \cite{BerestyckiPowell,RhodesVargas}. Applications of Theorem~\ref{thm:main} to the Gaussian Multiplicative Chaos will be considered in the sequel to this note.


\begin{thebibliography}{24}
		
		\bibitem{BerestyckiPowell}
		N.~Berestycki, E.~Powell,
		\textit{Gaussian free field and Liouville quantum gravity},
		preprint, arXiv:2404.16642.
		
		\bibitem{Berestycki2}
		N.~Berestycki,
		\textit{An elementary approach to Gaussian multiplicative chaos},
		Electron.~Commun.~Probab. \textbf{22} (2017), 1--12.
		
		\bibitem{Billingsley}
		P.~Billingsley,
		\textit{Convergence of Probability Measures},
		Wiley, 1968.
		
		\bibitem{Bockstein}
		M.~Bockstein,
		\textit{Un théorème de séparabilité pour les produits topologiques},
		Fundam.~Math. \textbf{35} (1948), 242--246.
		
		\bibitem{Bogachev}
		V.~I.~Bogachev,
		\textit{Foundations of Measure Theory}, 2nd~ed.,
		Regular and Chaotic Dynamics, Moscow--Izhevsk, 2006 (in Russian).
		
		\bibitem{Bufetov1}
		A.~I.~Bufetov,
		\textit{Gaussian multiplicative chaos for the sine-process},
		Uspekhi Mat.~Nauk \textbf{78}:6(474) (2023), 179--180 (in Russian).
		
		\bibitem{Bufetov2}
		A.~I.~Bufetov,
		\textit{The mean value of the multiplicative functional of the sine-process},
		Funktsional.~Anal.~i Ego Prilozh. \textbf{58}:2 (2024), 23--33 (in Russian).
		
		\bibitem{Bufetov3}
		A.~I.~Bufetov,
		\textit{The rate of convergence in the Kolmogorov--Smirnov metric in the central limit theorem of Soshnikov for the sine-process},
		Funktsional.~Anal.~i Ego Prilozh. \textbf{59}:2 (2025), 11--16 (in Russian).
		
		\bibitem{Chhaibi}
		R.~Chhaibi, J.~Najnudel,
		\textit{On the circle, Gaussian multiplicative chaos and Beta ensembles match exactly},
		to appear in JEMS, arXiv:1904.00578.
		
		\bibitem{FyodorovBouchaud}
		Y.~V.~Fyodorov, J.-P.~Bouchaud,
		\textit{Freezing and extreme value statistics in a random energy model with logarithmically correlated potential},
		J.~Phys.~A: Math.~Theor. \textbf{41}(37) (2008).
		
		\bibitem{FyodorovKeating}
		Y.~V.~Fyodorov, J.~P.~Keating,
		\textit{Freezing transitions and extreme values: random matrix theory and disordered landscapes},
		Philos.~Trans.~R.~Soc.~Lond.~Ser.~A Math.~Phys.~Eng.~Sci. \textbf{372} (2014), 20120503.
		
		\bibitem{FyodorovKhoruzhenkoSimm}
		Y.~V.~Fyodorov, B.~A.~Khoruzhenko, N.~J.~Simm,
		\textit{Fractional Brownian motion with Hurst index $H = 0$ and the Gaussian unitary ensemble},
		Ann.~Probab. \textbf{44}:4 (2016), 2980--3031.
		
		\bibitem{Gorbunov}
		S.~M.~Gorbunov,
		\textit{The rate of convergence in the central limit theorem for a determinantal point process with Bessel kernel},
		Mat.~Sb. \textbf{215}:12 (2024), 30--55 (in Russian).
		
		\bibitem{Zolotarev}
		V.~M.~Zolotarev,
		\textit{Metric distances in spaces of random variables and their distributions},
		Mat.~Sb. \textbf{101}(143):3(11) (1976), 416--454 (in Russian).
		
		\bibitem{Kahane}
		J.-P.~Kahane,
		\textit{Sur le chaos multiplicatif},
		Ann.~Sci.~Math.~Québec \textbf{9} (1985), 105--150.
		
		\bibitem{LacoinGMC}
		H.~Lacoin,
		\textit{Critical Gaussian multiplicative chaos revisited},
		Ann.~Inst.~H.~Poincaré Probab.~Statist. \textbf{60}:4 (2024), 2328--2351.
		
		\bibitem{LacoinRRV}
		H.~Lacoin, R.~Rhodes, V.~Vargas,
		\textit{Complex Gaussian multiplicative chaos},
		Commun.~Math.~Phys. \textbf{337} (2015), 569--632.
		
		\bibitem{LacoinUniv}
		H.~Lacoin,
		\textit{A universality result for subcritical complex Gaussian multiplicative chaos},
		Ann.~Appl.~Probab. \textbf{32}:1 (2022), 269--293.
		
		\bibitem{LambertOstrovskySimm}
		G.~Lambert, D.~Ostrovsky, N.~Simm,
		\textit{Subcritical multiplicative chaos for regularized counting statistics from random matrix theory},
		Commun.~Math.~Phys. \textbf{360} (2018), 1--54.
		
		\bibitem{NajnudelEtAl}
		J.~Najnudel, E.~Paquette, N.~Simm, T.~Vu,
		\textit{The Fourier coefficients of the holomorphic multiplicative chaos in the limit of large frequency},
		preprint, arXiv:2502.14863.
		
		\bibitem{NikulaSaksmanWebb}
		M.~Nikula, E.~Saksman, C.~Webb,
		\textit{Multiplicative chaos and the characteristic polynomial of the CUE: the $L^1$-phase},
		Trans.~Amer.~Math.~Soc. \textbf{373} (2020), 3905--3965.
		
		\bibitem{RhodesVargas}
		R.~Rhodes, V.~Vargas,
		\textit{Gaussian multiplicative chaos and applications: a review},
		Probab.~Surv. \textbf{11} (2014), 315--392.
		
		\bibitem{Shamov}
		A.~Shamov,
		\textit{On Gaussian multiplicative chaos},
		J.~Funct.~Anal. \textbf{270}:9 (2016), 3224--3261.
		
		\bibitem{WebbCUE}
		C.~Webb,
		\textit{The characteristic polynomial of a random unitary matrix and Gaussian multiplicative chaos: the $L^2$-phase},
		Electron.~J.~Probab. \textbf{20}(104) (2015), 1--21.
		
	\end{thebibliography}
\end{document}